\documentclass[12pt]{article}

\usepackage{mathtools}
\usepackage{amsmath,amsthm}
\usepackage{csquotes}

\usepackage{amssymb}

\usepackage{graphicx}

\usepackage{subfig}

\usepackage{mathpazo}

\usepackage{mathabx} 

\usepackage{soul}

\usepackage{xcolor}

\usepackage{cancel}

\usepackage{empheq}

\usepackage{lipsum}

\definecolor{myblue}{rgb}{.8, .8, 1}
  \newcommand*\mybluebox[1]{
    \colorbox{myblue}{\hspace{1em}#1\hspace{1em}}}

\usepackage[obeyspaces,hyphens,spaces]{url}

\usepackage{hyperref}
\hypersetup{
    colorlinks=true,
    linkcolor=blue,
    citecolor=blue,
    filecolor=magenta,
    urlcolor=cyan
}

\usepackage[
  open,
  openlevel=2,
  atend,
  numbered
]{bookmark}

\usepackage[capitalize, nameinlink, noabbrev]{cleveref}
\crefname{equation}{}{}
\crefname{chapter}{Chapter}{Chapters}
\crefname{item}{item}{items}
\crefname{figure}{Figure}{Figures}
\crefname{theorem}{Theorem}{Theorems}
\crefname{lemma}{Lemma}{Lemmas}
\crefname{proposition}{Proposition}{Propositions}
\crefname{corollary}{Corollary}{Corollarys}
\crefname{definition}{Definition}{Definitions}
\crefname{fact}{Fact}{Facts}
\crefname{example}{Example}{Examples}
\crefname{algorithm}{Algorithm}{Algorithms}
\crefname{remark}{Remark}{Remarks}
\crefname{note}{Note}{Notes}
\crefname{notation}{Notation}{Notations}
\crefname{case}{Case}{Cases}
\crefname{exercise}{Exercise}{Exercises}
\crefname{question}{Question}{Questions}
\crefname{claim}{Claim}{Claims}
\crefname{enumi}{}{}

\usepackage[top= 2cm, bottom = 2 cm, left = 2.2 cm, right= 2.2 cm]{geometry}

\usepackage{float}

\numberwithin{equation}{section}

\theoremstyle{plain}
\newtheorem{theorem}{Theorem}[section]

\newtheorem{corollary}[theorem]{Corollary}
\newtheorem{fact}[theorem]{Fact}
\newtheorem{lemma}[theorem]{Lemma}

\theoremstyle{definition}

\newtheorem{example}[theorem]{Example}

\newtheorem{remark}[theorem]{Remark}

\usepackage{enumitem}

\newcommand{\inte}{\ensuremath{\operatorname{int}}}

\newcommand{\conv}{\ensuremath{\operatorname{conv}}}

\newcommand{\Fix}{\ensuremath{\operatorname{Fix}}}
\newcommand{\Id}{\ensuremath{\operatorname{Id}}}

\newcommand{\floor}[1]{\left\lfloor #1 \right\rfloor}

\providecommand{\innp}[1]{\langle#1\rangle}

\newcommand{\scal}[2]{\left\langle{#1},{#2}  \right\rangle}

\newcommand{\NN}{\ensuremath{{{\mathbb N}}}}
\newcommand{\nnn}{\ensuremath{{n\in{\mathbb N}}}}
\newcommand{\kkk}{\ensuremath{{k\in{\mathbb N}}}}
\newcommand{\sss}{\ensuremath{{{\mathbb S}}}}
\newcommand{\RR}{\ensuremath{{{\mathbb R}}}}
\newcommand{\QQ}{\ensuremath{{{\mathbb{Q}}}}}
\newcommand{\menge}[2]{\big\{{#1}~\big |~{#2}\big\}}
\newcommand{\tmenge}[2]{\{{#1}~|~{#2}\}}

\begin{document}

\title{ \sffamily  Directional asymptotics of Fej\'er monotone sequences}

\author{
         Heinz H.\ Bauschke\thanks{
                 Mathematics, University of British Columbia, Kelowna, B.C.\ V1V~1V7, Canada.
                 E-mail: \href{mailto:heinz.bauschke@ubc.ca}{\texttt{heinz.bauschke@ubc.ca}}.},~
         Manish\ Krishan Lal\thanks{
                 Mathematics, University of British Columbia, Kelowna, B.C.\ V1V~1V7, Canada.
                 E-mail: \href{mailto:manish.krishanlal@ubc.ca}{\texttt{manish.krishanlal@ubc.ca}}.},~
         and Xianfu\ Wang\thanks{
                 Mathematics, University of British Columbia, Kelowna, B.C.\ V1V~1V7, Canada.
                 E-mail: \href{mailto:shawn.wang@ubc.ca}{\texttt{shawn.wang@ubc.ca}}.}
                 }

\date{June 29, 2021}

\maketitle

\begin{abstract}
The notion of Fej\'er monotonicity is instrumental in unifying the convergence
proofs of many iterative
 methods, such as the Krasnoselskii--Mann iteration,
the proximal point method, the Douglas-Rachford splitting algorithm,
 and many others.
In this paper, we present directionally asymptotical results of strongly
convergent subsequences
 of Fej\'er monotone sequences. We also provide examples to
show that the sets of directionally asymptotic cluster
 points can be large and
that weak convergence is needed in infinite-dimensional spaces.
\end{abstract}

{\small

}

\section{Introduction}

One of the most important tools in studying convergence of iterative methods 
in optimization and convex analysis
is Fej\'er monotonicity;
see, e.g., \cite[Chapters 5, 26, 28]{BC2017}, \cite{cegielski},
\cite{dao15},
\cite{PLC00}, \cite{PLC01}, \cite{PLC04}.
Recently,  among many
important advances,
Rockafellar showed in \cite{Rocky2021} that in a finite-dimensional Hilbert space
the sequences generated by the proximal point
algorithm enjoy directionally asymptotic properties.
In this paper, we study directional asymptotics
of Fej\'er monotone sequences in Hilbert spaces. Consequently, many iteration methods in \cite{BC2017} and
\cite{cegielski}, whose convergence analysis
relies on the Fej\'er monotonicity, have these directionally asymptotic behaviour.

%
%


The paper is organized as follows.
In \cref{sec:aux}, we provide some preliminary results on Fej\'er monotone sequences
useful in subsequent proofs.
Our main results on directional asymptotics of Fej\'er monotone sequences
are presented in \cref{sec:main}. In \cref{s:large:cluster}, we show that the sets of
directional asymptotics of Fej\'er monotone sequences can be large.
We conclude the paper with an infinite-dimensional example illustrating 
weak without strong convergence
in \cref{sec:shift}.

The notation that we employ is for the most part standard and follows
\cite{BC2017}; however, a partial
list is provided for the reader’s convenience.
Throughout this paper, we assume that
\begin{empheq}[box = \mybluebox]{equation*}
\text{$X$ is a real Hilbert space},
\end{empheq}
with inner product $\innp{\cdot,\cdot}$ and induced
 norm $\|\cdot\|$. We use  $\NN:=\{0, 1, 2, \ldots\}$ for set of natural numbers. 

Let $(x_{n})_\nnn$ be a sequence in $X$.
We denote the set of \emph{(weak) cluster points of $(x_n)_\nnn$} by 
\begin{empheq}[box = \mybluebox]{equation*}
  \mathcal{C}\big((x_n)_\nnn\big) := 
  \menge{x\in X}{\text{$x$ is the weak limit of some subsequence of $(x_n)_\nnn$}}. 
\end{empheq}
Of course, if $X$ is finite-dimensional, then 
$\mathcal{C}\big((x_n)_\nnn\big)$ is the same as the set of strong cluster points 
of $(x_n)_\nnn$. 
We write $x_{n}\rightarrow x$ if $(x_{n})_\nnn$ converges strongly to $x$, and
$x_{n}\rightharpoonup x$ if $(x_{n})_\nnn$ converges weakly to $x$.
Let $C$ be a subset of $X$ and let $(x_{n})_\nnn$ be a sequence in $X$.
Then $(x_{n})_\nnn$ is \emph{Fej\'er monotone with respect to $C$} if
$$(\forall c\in C)(\forall \nnn)\ \|x_{n+1}-c\|\leq \|x_{n}-c\|,$$
and we also call $C$ a \emph{Fej\'er monotone set of $(x_{n})_\nnn$}.
The corresponding \emph{support function} of $C$ is 
defined by $\sigma_C(x):=\sup\scal{C}{x}$ while
the corresponding \emph{distance function} is 
$d_C(x):=\inf\|C-x\|$, for every $x\in X$. The 
\emph{polar cone of $C$} is 
$C^{\ominus}:=\menge{u\in X}{\sigma_{C}(u)\leq 0}$;
note that if $z\in C$, then
$N_C(z)=(C-z)^\ominus$ is the \emph{normal cone} of $C$ at $x$.
For a set-valued monotone operator $A:X\rightrightarrows X$, 
the corresponding \emph{resolvent} is 
$J_{A}:=(\Id +A)^{-1}$.
Finally, the \emph{unit sphere} is abbreviated by 
\begin{empheq}[box = \mybluebox]{equation*}
  \sss :=\menge{x\in X}{\|x\|=1}. 
\end{empheq}

\section{Auxiliary results}
\label{sec:aux}

We start with some preparatory results on Fej\'er monotone sequences.

\begin{lemma}\label{l:fejer} Let $(x_{n})_\nnn$ be a sequence in $X$. Then the following hold:
\begin{enumerate}
\item\label{i:largest}
 The largest Fej\'er monotone set of $(x_{n})_\nnn$ is the (possibly empty) 
 closed convex set
$$\bigcap_\nnn\menge{z\in X}{2\scal{x_{n}-x_{n+1}}{z}\leq \|x_n\|^2-\|x_{n+1}\|^2},$$
and is closed convex.
\item\label{i:convexhull}
 If $C$ is a Fej\'er monotone set of $(x_{n})_\nnn$, then
$\overline{\conv}\, C$, the closed convex hull of $C$, is a Fej\'er monotone set of $(x_n)_{\nnn}$.
\item\label{i:union} If $C_{1}, C_{2}$ are Fej\'er monotone sets of $(x_{n})_{\nnn}$, then
$C_{1}\cup C_{2}$ is a Fej\'er monotone set of $(x_{n})_{\nnn}$.
\item\label{i:subfejer}
If $C$ is a Fej\'er monotone set of $(x_{n})_\nnn$, then $C$ is a Fej\'er monotone set
of every subsequence of $(x_{n})_\nnn$.
\end{enumerate}
\end{lemma}
\begin{proof}
\cref{i:largest}: Let $z\in X$ and $\nnn$. 
Then 
$\|x_{n+1}-z\|\leq \|x_{n}-z\|$
$\Leftrightarrow$ 
$\|x_{n+1}-z\|^2\leq \|x_{n}-z\|^2$
$\Leftrightarrow$ 
$\|x_{n+1}\|^2 + \|z\|^2-2\scal{x_{n+1}}{z} 
\leq \|x_{n}\|^2 + \|z\|^2-2\scal{x_{n}}{z}$
$\Leftrightarrow$ 
$2\scal{x_{n}-x_{n+1}}{z}\leq \|x_n\|^2-\|x_{n+1}\|^2$.
\cref{i:convexhull}\&\cref{i:union}: These follow from \cref{i:largest}.
\cref{i:subfejer}: Obvious from the definition of Fej\'er monotonicity. 
\end{proof}

\begin{lemma}\label{l:fejerset}
Let $C$ be a nonempty closed convex subset of $X$, let $\widebar{z}\in X$, and
let $(x_{n})_\nnn$ be a sequence in $X$.
Suppose that $(x_{n})_\nnn$ is Fej\'er monotone with respect to $C$. 
Then the following hold for all $n,m$ in $\NN$ such that $m\geq n+1$:
\begin{subequations}
\label{e:210625a}
\begin{align}
\label{i:f:consecutive}
(\forall z\in C)
\quad 
\scal{x_{n}-x_{n+1}}{z-\widebar{z}} 
&\leq \tfrac{1}{2}\big(\|x_n-\widebar{z}\|^2 - \|x_{n+1}-\widebar{z}\|^2\big) \\
\label{i:f:consecutive+}
&= \scal{x_{n+1}-\widebar{z}}{x_{n}-x_{n+1}}+\tfrac{1}{2}\|x_{n}-x_{n+1}\|^2\\
\label{i:f:consecutive++}
&\leq \|x_{n+1}-\widebar{z}\|\|x_{n}-x_{n+1}\|+\tfrac{1}{2}\|x_{n}-x_{n+1}\|^2
\end{align}
\end{subequations}
and 
\begin{align}\label{e:limit} 
(\forall z\in C) \quad \scal{x_{n}-x_{m}}{z-\widebar{z}} &
\leq \tfrac{1}{2}(\|x_{n}-\widebar{z}\|^2-\|x_{m}-\widebar{z}\|^2).
\end{align}
\end{lemma}
\begin{proof}
Let $z\in C$ and let $k\in\{n,n+1,\ldots,m\}$. 
Because $(x_n)_\nnn$ is Fej\'er monotone with respect to $C$, we have 
\begin{align*}
0 &\leq \|x_k-z\|^2 - \|x_{k+1}-z\|^2\\
&= \|x_k-x_{k+1}\|^2 + 2\scal{x_k-x_{k+1}}{x_{k+1}-z}\\
&= \|x_k-x_{k+1}\|^2 + 2\scal{x_k-x_{k+1}}{(x_{k+1}-\widebar{z})-(z-\widebar{z})}.
\end{align*}
Therefore,
\begin{subequations}
\label{e:tele}
\begin{align}
\scal{x_k-x_{k+1}}{z-\widebar{z}}
&\leq
\tfrac{1}{2}\|x_k-x_{k+1}\|^2 + \scal{x_k-x_{k+1}}{x_{k+1}-\widebar{z}}\\
&= 
\tfrac{1}{2}\big(\|x_k-\widebar{z}\|^2 - \|x_{k+1}-\widebar{z}\|^2\big) 
\end{align}
\end{subequations}
which yields \cref{i:f:consecutive} and \cref{i:f:consecutive+}.
Next, \cref{i:f:consecutive++} is just Cauchy-Schwarz. 
Finally, \cref{e:limit} follows by summing 
\cref{e:tele} from $k=n$ to $k=m$ and telescoping. 
\end{proof}

We now localize the set of weak cluster points of a sequence.

\begin{lemma}\label{l:polarset:f}
 Let $C$ be a nonempty subset of $X$ and let $(x_{n})_\nnn$ be a sequence in $X$. 
 Suppose that
$$\varlimsup_{n\to\infty}\sigma_{C}(x_{n})\leq 0.$$
Then 
$$\mathcal{C}\big((x_n)_\nnn\big) \subseteq C^\ominus.$$
\end{lemma}
\begin{proof}
Suppose that 
$x\in \mathcal{C}\big((x_{n})_\nnn\big)$ and to the contrary that 
$x\notin C^\ominus$. 
Then $\sigma_C(x)>0$ and there exists 
a weakly convergent subsequence $(x_{k_n})_\nnn$ of $(x_n)_\nnn$ such that 
$x_{k_n}\rightharpoonup x$. 
The weak lower semicontinuity of $\sigma_C$ now implies
$$0<\sigma_{C}(x)\leq \varliminf_{n\to\infty}\sigma_{C}(x_{k_{n}})
\leq\varlimsup_{n\rightarrow\infty}\sigma_{C}(x_{k_{n}})\leq
\varlimsup_{n\rightarrow\infty}\sigma_{C}(x_{n})\leq 0,$$
which is absurd!
\end{proof}

\begin{lemma}\label{l:distance} 
Suppose that $X$ is finite-dimensional,
let $C$ be a nonempty closed subset of $X$, and let $(x_{n})_\nnn$ be a bounded sequence
in $X$. Then
\begin{equation*}
d_C(x_n)\to 0 
\;\;\Leftrightarrow\;\;
\mathcal{C}\big((x_n)_\nnn \big) \subseteq C.
\end{equation*}
\end{lemma}
\begin{proof}
``$\Rightarrow$'': 
Let $x$ be a cluster point in $\mathcal{C}\big((x_{n})_\nnn\big)$, say
$x_{k_{n}}\rightarrow x$. 
The continuity of $d_C$ and the assumption yield
$$d_{C}(x)=\lim_{n\to\infty}d_{C}(x_{k_{n}})=
\lim_{n\to\infty}d_{C}(x_{n})=0.$$
Hence $x\in C$ because $C$ is closed. 

``$\Leftarrow$'':
Suppose to the contrary that
$\varlimsup_{n\to\infty} d_{C}(x_{n})>0$. 
Then there exists a subsequence $(x_{k_n})_\nnn$ of $(x_n)_\nnn$ such that 
\begin{equation}
\label{e:absurd}
  \lim_{n\to\infty}d_{C}(x_{k_{n}})=\varlimsup_{n\to\infty} d_{C}(x_{n})>0.
\end{equation}
Recall that $(x_n)_\nnn$ is bounded and $X$ is finite-dimensional.
Using Bolzano-Weierstrass and after passing to another subsequence 
and relabeling, we may and do assume that 
$x_{k_n}\to x$. 
By assumption, $x\in C$. 
But then $d_C(x_{k_n}) \to d_C(x)=0$ which contradicts \cref{e:absurd}.
\end{proof}

We end this section with results on linear isometries. 

\begin{lemma}\label{l:skew}
 Let $A:X\to X$ be a linear isometry. 
Then the following hold: 
\begin{enumerate}
\item\label{i:zset} $A$ is injective. 
\item\label{i:resol2} 
If $T := \tfrac{1}{2}\Id + \tfrac{1}{2}A$, then 
$(\forall x\in X)$ $Tx \perp (x-Tx)$. 
 \item\label{i:resol1} 
If $A^*=A^{-1}=-A$, then 
$J_{A}=\tfrac{1}{2}\Id-\tfrac{1}{2}A$ and 
$(\forall x\in X)$ $J_Ax \perp (x-J_Ax)$. 
\end{enumerate}
\end{lemma}
\begin{proof}
Let $x\in X$. 
\cref{i:zset}: If $Ax=0$, then $0=\|Ax\|=\|x\|$ and so $x=0$. 

\cref{i:resol2}: 
Note that $\Id-T = \tfrac{1}{2}\Id-\tfrac{1}{2}$. Hence 
\begin{equation*}
4\scal{Tx}{x-Tx} = 
\scal{x+Ax}{x-Ax} = \|x\|^2-\|Ax\|^2=0.
\end{equation*}
\cref{i:resol1}: 
Clearly, $\pm A$ is monotone and $A^2 = -\Id$.
Hence, \cite[Proposition~2.10]{84} yields $J_A=\tfrac{1}{2}\Id-\tfrac{1}{2}A=
\Id-T$, where $T=\tfrac{1}{2}\Id+\tfrac{1}{2}A$.
Now apply \cref{i:resol2}. 
\end{proof}

\begin{corollary}\label{t:perp}
Let $A:X\to X$ be a linear operator such that $A^* = A^{-1} = -A$, 
let $x_0\in X\smallsetminus\{0\}$, and set 
$$(\forall \nnn)\quad x_{n+1}:=J_{A}x_{n} = \tfrac{1}{2}x_n-\tfrac{1}{2}Ax_n.$$
Then $x_n\to 0$ and $(\forall \nnn)$
$x_{n+1}\neq x_n$, and 
\begin{equation}\label{e:orthogonal}
(\forall \nnn)\ \scal{\frac{x_{n+1}}{\|x_{n+1}\|}}{\frac{x_n-x_{n+1}}{\|x_n-x_{n+1}\|}}=0.
\end{equation}
\end{corollary}
\begin{proof} 
Clearly, $A$ is a maximally monotone isometry.
The formula for $J_Ax_n$ is a consequence of \cref{l:skew}\cref{i:resol1} 
which also yields \cref{e:orthogonal} after we prove that the 
quotients are well defined which we do next. 
Let $x\in X$.
Then 
$(\Id+A)^{-1}x = J_Ax=0$
$\Leftrightarrow$
$x=(\Id+A)(0) = 0$ and 
$(\Id+A)^{-1}x = J_Ax=x$
$\Leftrightarrow$
$x=(\Id+A)x$
$\Leftrightarrow$
$Ax=0$
$\Leftrightarrow$
$x=0$. We have shown that 
if $x\neq 0$, then 
$J_Ax\neq 0$ and $J_Ax\neq x$. 
A straightforward induction yields 
$(\forall\nnn)$ $x_n\neq 0$ and $x_{n+1}\neq x_n$, as claimed. 
Finally, \cite[Corollary 1.2]{bruck77} implies that $x_n \to 0$. 
\end{proof}

\begin{remark}
When $X=\RR^2$ and
$$A=\begin{pmatrix}
0 & 1\\
-1 & 0
\end{pmatrix},
$$
then \cref{t:perp} recovers \cite[the example on page~11]{Rocky2021}.
Note that (see \cite[page~206]{kreyszig}) a linear isometry need not be
surjective. 
\end{remark}

\section{Directional asymptotics of Fej\'er monotone sequences}
\label{sec:main}

We are now ready for our main results on the directionally asymptotic behaviour of 
Fej\'er monotone sequences. 
 The proofs significantly extend the reach of those brought to light by 
 Rockafellar in \cite[Theorem~2.3]{Rocky2021}.

\begin{theorem}\label{t:main}
Let $(x_{n})_\nnn$ be a sequence in $X$ that is Fej\'er monotone with respect to 
some nonempty closed convex subset $Z$ of $X$. 
Suppose that $x_n\to \widebar{z}\in X$ and that
$(\forall \nnn)$ $x_{n+1}\neq x_{n}\neq\widebar{z}$.
Then 
\begin{equation}
\label{e:polar1}
\mathcal{C}\bigg(\Big(\frac{x_n-x_{n+1}}{\|x_n-x_{n+1}\|} \Big)_\nnn\bigg)
\cup 
\mathcal{C}\bigg(\Big(\frac{x_n-\widebar{z}}{\|x_n-\widebar{z}\|} \Big)_\nnn\bigg)
\subseteq 
(Z-\widebar{z})^{\ominus};
\end{equation}
in particular, if $\widebar{z}\in Z$, then we may replace $(Z-\widebar{z})^\ominus$ by 
$N_Z(\widebar{z})$ in  \cref{e:polar1}.
\end{theorem}
\begin{proof}
Let $\nnn$. 
Taking the supremum over $z\in Z$ in \cref{e:210625a} yields
\begin{align}
\label{e:210625b}
\sigma_{Z-\widebar{z}}(x_{n}-x_{n+1})
& \leq \|x_{n+1}-\widebar{z}\|\|x_{n}-x_{n+1}\|+\tfrac{1}{2}\|x_{n}-x_{n+1}\|^2.
\end{align}
Dividing \cref{e:210625b} by $\|x_{n}-x_{n+1}\|$ and using the positive
homogeneity of $\sigma_{Z-\widebar{z}}$, we have
\begin{equation}
\label{e:210625c}
\sigma_{Z-\widebar{z}}\Big(\frac{x_n-x_{n+1}}{\|x_n-x_{n+1}\|} \Big)\leq \|x_{n+1}-\widebar{z}\|
+\tfrac{1}{2}\|x_{n}-x_{n+1}\|.
\end{equation}
Because $x_n\to\widebar{z}$ and so $x_n-x_{n+1}\to 0$, we let 
$n\to\infty$ in \cref{e:210625c} to learn that 
\begin{equation}
\label{e:210625d}
\varlimsup_{n\to\infty}
\sigma_{Z-\widebar{z}}\Big(\frac{x_n-x_{n+1}}{\|x_n-x_{n+1}\|} \Big)\leq 0.
\end{equation}
Combining \cref{e:210625d} with \cref{l:polarset:f},
we obtain 
\begin{equation}
\label{e:210625e}
\mathcal{C}\bigg(\Big(\frac{x_n-x_{n+1}}{\|x_n-x_{n+1}\|} \Big)_\nnn\bigg)
\subseteq 
(Z-\widebar{z})^{\ominus}.
\end{equation}

Next, let $m\geq n+1$. 
Taking the supermum over $z\in Z$ in \cref{e:limit} yields
\begin{align}
\label{e:210625f}
\sigma_{Z-\widebar{z}}(x_{n}-x_{m}) &\leq
\tfrac{1}{2}\big(\|x_{n}-\widebar{z}\|^2-\|x_{m}-\widebar{z}\|^2\big).
\end{align}
Passing to the limit as $m\to\infty$ and using the lower semicontinuity of
$\sigma_{Z-\widebar{z}}$ in \cref{e:210625f}, we obtain
\begin{equation}
\label{e:210625g}
\sigma_{Z-\widebar{z}}(x_{n}-\widebar{z})\leq 
\varliminf_{m\rightarrow\infty} \sigma_{Z-\widebar{z}}(x_{n}-x_{m})
\leq \tfrac{1}{2}\|x_{n}-\widebar{z}\|^2.
\end{equation}
Dividing \cref{e:210625g} by 
$\|x_{n}-\widebar{z}\|$ and
using the positive homogeneity of $\sigma_{Z-\widebar{z}}$, we have 
\begin{equation}
\label{e:210625h}
\sigma_{Z-\widebar{z}}\Big(\frac{x_n-\widebar{z}}{\|x_n-\widebar{z}\|} \Big)
\leq \tfrac{1}{2}\|x_{n}-\widebar{z}\|.
\end{equation}
Because $x_n\to\widebar{z}$, we let $n\to\infty$ in \cref{e:210625h} and get 
\begin{equation}
\label{e:210625i}
\varlimsup_{n\rightarrow\infty}\sigma_{Z-\widebar{z}}\Big(\frac{x_n-\widebar{z}}{\|x_n-\widebar{z}\|} \Big)\leq 0. 
\end{equation}
Combining \cref{e:210625i} with our trusted \cref{l:polarset:f}, we obtain 
\begin{equation}
  \label{e:210625j}
  \mathcal{C}\bigg(\Big(\frac{x_n-\widebar{z}}{\|x_n-\widebar{z}\|} \Big)_\nnn\bigg)
  \subseteq 
  (Z-\widebar{z})^{\ominus}.
  \end{equation}
Altogether, \cref{e:210625e} and \cref{e:210625j} imply \cref{e:polar1}. 
\end{proof}

Although ostensibly more general, the following result is actually an easy consequence of 
\cref{t:main}: 

\begin{corollary}\label{t:notconsecutive}
Let $(x_{n})_\nnn$ be a sequence in $X$ that is Fej\'er monotone with respect to 
some nonempty closed convex subset $Z$ of $X$. 
Suppose that $(x_{k_n})_\nnn$ is a subsequence of $(x_n)_\nnn$ such that 
$x_{k_n}\to \widebar{z}\in X$ and that
$(\forall \nnn)$ $x_{k_{n+1}}\neq x_{k_n}\neq\widebar{z}$.
Then 
\begin{equation}
\label{e:210625k}
\mathcal{C}\bigg(\Big(\frac{x_{k_n}-x_{k_{n+1}}}{\|x_{k_n}-x_{k_{n+1}}\|} \Big)_\nnn\bigg)
\cup 
\mathcal{C}\bigg(\Big(\frac{x_{k_n}-\widebar{z}}{\|x_{k_n}-\widebar{z}\|} \Big)_\nnn\bigg)
\subseteq 
(Z-\widebar{z})^{\ominus};
\end{equation}
in particular, if $\widebar{z}\in Z$, then we may replace $(Z-\widebar{z})^\ominus$ by 
$N_Z(\widebar{z})$ in \cref{e:210625k}.
\end{corollary}
\begin{proof}
Recalling \cref{l:fejer}\cref{i:subfejer}, 
we simply apply \cref{t:main} to $(x_{k_n})_\nnn$. 
\end{proof}

When $X$ is finite-dimensional, we have the following two nice results:

\begin{corollary}\label{t:finite:dim}
Suppose that $X$ is finite-dimensional.
Let $(x_{n})_\nnn$ be a sequence in $X$ that is Fej\'er monotone with respect to 
some nonempty closed convex subset $Z$ of $X$. 
Suppose that $(x_{k_n})_\nnn$ is a subsequence of $(x_n)_\nnn$ such that 
$x_{k_n}\to \widebar{z}\in Z$ and that
$(\forall \nnn)$ $x_{k_{n+1}}\neq x_{k_n}\neq\widebar{z}$.
Then 
\begin{equation*}
\mathcal{C}\bigg(\Big(\frac{x_{k_n}-x_{k_{n+1}}}{\|x_{k_n}-x_{k_{n+1}}\|} \Big)_\nnn\bigg)
\cup 
\mathcal{C}\bigg(\Big(\frac{x_{k_n}-\widebar{z}}{\|x_{k_n}-\widebar{z}\|} \Big)_\nnn\bigg)
\subseteq 
\sss \cap N_Z(\widebar{z});
\end{equation*}
equivalently, 
\begin{equation*}
\lim_{n\rightarrow\infty}d_{\sss\cap N_{Z}(\widebar{z})}
\Big(\frac{x_{k_n}-x_{k_{n+1}}}{\|x_{k_n}-x_{k_{n+1}}\|} \Big)=0
\quad \text{ and }\quad
\lim_{n\rightarrow\infty}d_{\sss\cap N_{Z}(\widebar{z})}
\Big(\frac{x_{k_n}-\widebar{z}}{\|x_{k_n}-\widebar{z}\|}\Big)=0.
\end{equation*}
\end{corollary}
\begin{proof} 
  Combine \cref{t:notconsecutive} with  \cref{l:distance}.
\end{proof}


\begin{corollary}[{\bf no zigzagging}]
\label{c:nozigzag}
Suppose that $X$ is finite-dimensional
and let $(x_n)_\nnn$ be a sequence that is 
Fej\'er monotone with respect to some closed convex subset $Z$ of $X$.
Suppose that $x_n\to \widebar{z}\in Z$, that 
$(\forall\nnn)$ $x_{n+1}\neq x_n\neq\widebar{z}$, 
that $\inte Z\neq\varnothing$, and that
$N_Z(\widebar{z})$ is a ray.
Then
\begin{equation}
\label{e:nozigzag}
\lim_{n\to\infty}\frac{x_n-x_{n+1}}{\|x_n-x_{n+1}\|}
=
\lim_{n\to\infty}\frac{x_n-\widebar{z}}{\|x_n-\widebar{z}\|} 
\in \sss\cap N_Z(\widebar{z}).
\end{equation}
\end{corollary}
\begin{proof}
Clear from \cref{t:finite:dim} because $\sss \cap N_Z(\bar{z})$ is a singleton
when $N_Z(\widebar{z})$ is a ray. 
\end{proof}

\section{Large sets of directionally asymptotic cluster points}\label{s:large:cluster}

In this section, we give an example illustrating that the sets of directionally asymptotic
cluster points can be large. It also shows that
without the interiority assumption in \cref{c:nozigzag},
\cref{e:nozigzag} can go quite wrong.

We start with a fact from real analysis: 

\begin{fact}[{\bf Dirichlet}]
{\rm (See, e.g., \cite[page~88]{sohrab})}
\label{l:dirich}
Let $\alpha\in\RR\smallsetminus \QQ$. Then the set
$\tmenge{n\alpha-\floor{n\alpha}}{\nnn}$ is dense in $[0,1]$.
\end{fact}

For the remainder of this section, $R_\alpha$ denotes
the counterclockwise rotator in the Euclidean plane by $\alpha$.

\begin{example}
Suppose that $X=\RR^2$, 
let  $0<\theta\notin \tfrac{1}{2}\pi\NN$, 
Then 
$T := \tfrac{1}{2}\Id+\tfrac{1}{2}R_{2\theta}
= \cos(\theta)R_{\theta}$ is firmly nonexpansive,
with $Z := \Fix T = \{0\}$. 
Let $x_0\in X\smallsetminus\{0\}$, and set 
\begin{equation}\label{e:iterate}
(\nnn)\quad \ x_{n+1}:=T x_{n}.
\end{equation}
Then $x_n\to \widebar{z} := 0$, and 
$(\forall \nnn)$ 
$x_{n+1}\neq x_n\neq \widebar{z}$
and 
$\scal{x_{n}-x_{n+1}}{x_{n+1}}=0$. 
Moreover, we have the following dichotomoy:
\begin{enumerate}
\item $\theta\in  2\pi\QQ$ and
\begin{equation*}
\mathcal{C}\bigg(\Big(\frac{x_n-x_{n+1}}{\|x_n-x_{n+1}\|} \Big)_\nnn\bigg)
\cup 
\mathcal{C}\bigg(\Big(\frac{x_n}{\|x_n\|} \Big)_\nnn\bigg)
\quad \text{is a \emph{finite} subset of $\sss$.}
\end{equation*}
\item $\theta\notin2\pi\QQ$ and 
\begin{equation*}
\mathcal{C}\bigg(\Big(\frac{x_n-x_{n+1}}{\|x_n-x_{n+1}\|} \Big)_\nnn\bigg)
=
\mathcal{C}\bigg(\Big(\frac{x_n}{\|x_n\|} \Big)_\nnn\bigg)
=\sss. 
\end{equation*}
\end{enumerate}
\end{example}

\begin{proof}
Because $\theta\notin 2\pi\mathbb{Z}$, we have $\Fix T=\{0\}$.
Note that
$$T^n=(\cos\theta)^n R_{n\theta}=(\cos\theta)^n
\begin{pmatrix}
\cos n \theta & -\sin n\theta\\
\sin n\theta & \cos n\theta
\end{pmatrix}
\;\;\text{ and }
\;\; x_{n}=T^n x_{0}.$$
Since  $0<|\cos\theta|<1$ and $R_{n\theta}$ is an isometry,
we have $\|T^n\|\rightarrow 0$. Thus 
$x_{n}\rightarrow 0$. 
By \cref{l:skew}\cref{i:resol2}, we have
$\scal{x_{n}-x_{n+1}}{x_{n+1}} =0.$
Moreover, $\|x_{n+1}\|=\|\cos\theta R_{\theta} x_{n}\|=|\cos\theta |\|x_{n}\|<\|x_{n}\|$ because
$x_{0}\not =0$,
so $(\forall \nnn)\ x_{n+1}\neq x_{n}$ and $x_{n}\neq 0$.

To study the set of cluster points of
$(T^{n}x_{0}/\|T^nx_{0}\|)_\nnn$, we consider two cases.

\noindent Case 1: $\cos\theta>0$. We have
\begin{align}
\frac{T^{n}x_{0}}{\|T^n x_{0}\|}& =R_{n\theta}\frac{x_{0}}{\|x_{0}\|}
=\begin{pmatrix}\label{e:finitec}
\cos n \theta & -\sin n\theta\\
\sin n \theta & \cos n\theta
\end{pmatrix}
\frac{x_{0}}{\|x_{0}\|}.
\end{align}
We proceed with two subcases.

Subcase 1: $\frac{\theta}{2\pi}\in \QQ$. Then
$\theta=2\pi \frac{k}{l}$ with $k,l\in\NN$ and $l\neq 0$, and
$\cos n\theta=\cos \frac{n}{l}(2k\pi)$ and
$\sin n\theta=\sin \frac{n}{l}(2k\pi)$. By using $n=ml, ml+1, \ldots, ml+l-1$ with $m\in\NN$, the set
$$\menge{R_{n\theta}}{\nnn}=\menge{R_{t2k\pi/l}}{t=0,\ldots, l-1}.$$
The sequence $(R_{n\theta})_{\nnn}$ has at most $l$ cluster points. 
From \cref{e:finitec} we see that
$(T^{n}x_{0}/\|T^n x_{0}\|)_\nnn$ has at most $l$ cluster points.
In fact, if $k=2$, then there are precisely $l$ cluster points. 

Subcase 2: $\frac{\theta}{2\pi}\not\in \QQ$.
Then $\theta=2\pi \alpha$ with $\alpha\in \RR_{++}\smallsetminus \QQ$, and
$$\cos n\theta=\cos n(2\pi\alpha)=\cos (n\alpha)(2\pi)=\cos(n\alpha-\floor{n\alpha})(2\pi),\text{ and }$$
$$\sin n\theta=\sin n(2\pi\alpha)=\sin (n\alpha)(2\pi)=\sin (n\alpha-\floor{n\alpha})(2\pi),$$
 By \cref{l:dirich}, $\tmenge{n\alpha-\floor{n\alpha}}{\nnn}$ is dense in $[0,1]$. Hence
the set of cluster points of ${\tmenge{R_{n\theta}}{\nnn}}$ is
$\tmenge{R_{\beta}}{\beta\in [0,2\pi]}.$
By \cref{e:finitec}, the set of cluster points of
$(T^{n}x_{0}/\|T^n x_{0}\|)_\nnn$ is $\sss$.

\noindent
Case 2: $\cos\theta<0$. We have
\begin{align}\label{e:quotient}
\frac{T^{n}x_{0}}{\|T^n x_{0}\|}& =(-1)^n R_{n\theta}\frac{x_{0}}{\|x_{0}\|}=(-1)^{n}\begin{pmatrix}
\cos n \theta & -\sin n\theta\\
\sin n \theta & \cos n\theta
\end{pmatrix}
\frac{x_{0}}{\|x_{0}\|}.
\end{align}
We proceed with two subcases.

Subcase 1: $\frac{\theta}{2\pi}\in \QQ$. Due to $(-1)^{n}$, we have to consider
$n$ being even and odd.
When $n$ is even, write $n=2k$ with $k\in\NN$,
$$R_{n\theta}=\begin{pmatrix}
\cos k (2\theta) & -\sin k(2\theta)\\
\sin k (2\theta) & \cos k (2\theta)
\end{pmatrix}.
$$
Since $\frac{2\theta}{2\pi}\in \QQ$, similar arguments as in Case 1 subcase 1
show that
$(R_{k(2\theta)})_{k\in\NN}$ has a finite number of cluster points.
When $n$ is odd, let $n=2k+1$ and $\theta=m2\pi/l$ with $k, l, m\in\NN$ and $l\neq 0$. 
If we set
$k=tl+s$ for $t, s\in\NN$ and $0\leq s\leq l-1$, then
$$(2k+1)\frac{m2\pi}{l}=\frac{2(tl+s)+1}{l}m2\pi=\bigg(2t+\frac{2s+1}{l}\bigg)m2\pi$$
so that
$$\cos(2k+1)\theta=\cos\bigg(\frac{2s+1}{l}m2\pi\bigg), \quad \sin(2k+1)\theta=
\sin\bigg(\frac{2s+1}{l}m2\pi\bigg)$$
where $0\leq s\leq l-1$. This shows that when $n$ is odd, we have most $l$ cluster points.
Combining the even and odd cases, the set
$\menge{R_{n\theta}}{\nnn}$ has at most a finite number of cluster points, so is
$(T^n x_{0}/\|T^n x_{0}\|)_\nnn$ by \cref{e:quotient}.

Subcase 2: $\frac{\theta}{2\pi}\not\in \QQ$. Put $\theta=\alpha (2\pi)$ with 
$\alpha\not\in \QQ$.
When $n$ is even, write $n=2k$ with $k\in\NN$.
Then
$R_{n\theta}=R_{k(2\theta)}$, similar arguments as in Case 1 subcase 2 show that
the set of cluster points
of $\menge{R_{k(2\theta)}}{k\in\NN}$ is $\menge{R_{\beta}}{0\leq \beta\leq 2\pi}$,
because $\frac{2\theta}{2\pi}\in\RR_{++}\smallsetminus \QQ$.

When $n$ is odd, write $n=2k+1$ with $k\in \NN$. Since that
$$\cos[(2k+1)\alpha2\pi] =\cos[k(2\alpha)2\pi+2\alpha\pi]=
\cos[(k(2\alpha)-\floor{k(2\alpha)})2\pi+2\alpha\pi],\text{ and} $$
$$\sin[(2k+1)\alpha2\pi] =\sin[k(2\alpha)2\pi+2\alpha\pi]=
\sin[(k(2\alpha)-\floor{k(2\alpha)})2\pi+2\alpha\pi],$$
and that
$$\menge{k(2\alpha)-\floor{k(2\alpha)}}{k\in\NN} \text{ is dense in $[0,1]$}$$
we see that the set of cluster points of
$\tmenge{R_{(2k+1)\theta}}{k\in\NN}$ is
$\tmenge{R_{\beta}}{2\alpha\pi\leq\beta\leq 2\alpha\pi+2\pi}.$
Hence, in both cases the set of cluster points of $(T^{n}x_{0}/\|T^{n}x_{0}\|)_\nnn$
is $\sss$.

Finally, we consider the set of cluster points of
$$\Big(\frac{x_n-x_{n+1}}{\|x_n-x_{n+1}\|} \Big)_\nnn.$$
Now
\begin{align*}
\Id-T &=\frac{\Id-R_{2\theta}}{2}=\sin\theta \begin{pmatrix}
\sin \theta & \cos \theta\\
-\cos \theta & \sin \theta
\end{pmatrix}\\
& =\sin\theta \begin{pmatrix}
\cos(\theta+3\pi/2) & -\sin (\theta+3\pi/2)\\
\sin (\theta+3\pi/2) & \cos (\theta+3\pi/2)
\end{pmatrix}
=\sin\theta R_{(\theta+3\pi/2)},
\end{align*}
so that
\begin{align*}
\frac{x_{n}-x_{n+1}}{\|x_{n}-x_{n+1}\|} &=\frac{(\Id-T)x_{n}}{\|(\Id-T)x_{n}\|}\\
&=\begin{cases}
R_{(\theta+3\pi/2)} \frac{x_{n}}{\|x_{n}\|},& \text{ if $\sin\theta>0$;}\\
-R_{(\theta+3\pi/2)}\frac{x_{n}}{\|x_{n}\|}, & \text{ if $\sin\theta <0$.}
\end{cases}
\end{align*}
Then the set of cluster points of
$$\bigg(\frac{x_{n}-x_{n+1}}{\|x_{n}-x_{n+1}\|}\bigg)_{\nnn}$$ is just $\pm R_{(\theta+3\pi/2)}$ rotations
of the set of cluster points of
$(x_{n}/\|x_{n}\|)_{\nnn}$.
Consequently,
the set of cluster points of
$((x_{n}-x_{n+1})/\|x_{n}-x_{n+1}\|)_{\nnn}$
is a finite set if
$\frac{\theta}{2\pi}\in\QQ$; and
is $\sss$ if
$\frac{\theta}{2\pi}\not\in\QQ$.
\end{proof}

\begin{remark}
We do not consider the case when $\theta\in\tfrac{1}{2}\pi\NN$ 
because then $T=\Id$ or $0$  in which case one has 
finite convergence of $(x_n)_\nnn$. 
\end{remark}

\section{Missing the sphere: an infinite-dimensional example}

\label{sec:shift}
It is interesting to ask whether in infinite-dimensional Hilbert spaces
the nonempty sets of weak cluster points in \cref{t:main} lie in
the sphere $\sss$. 
It turns out that the answer is negative, and the sequence provided
is obtained by iterating a resolvent. To this end, 
we assume in this section that 
\begin{equation*}
	X = \ell^2\big(\{1,2,\ldots\}\big),
\end{equation*}
with the standard Schauder basis
$e_1:=(1,0,0,\ldots), e_2:=(0,1,0,0,\ldots)$, and so on.
We define the \emph{right-shift operator} by 
\begin{equation*}
R \colon X\to X\colon (\xi_1,\xi_2,\ldots)
\mapsto (0,\xi_1,\xi_2,\ldots),
\end{equation*}
Then $R$ is a linear isometry with $\Fix R = \{0\}$. 
We shall also require the following classical identity

\begin{fact}[{\bf Vandermonde's identity}] {\rm (See \cite[Section~5.1]{GKP}.)}
\label{l:vander}
Let $m, n, r$ be in $\NN$. Then
$$\binom{m+n}{r}=\sum_{k=0}^{r}\binom{m}{r}\binom{n}{r-k}.$$
\end{fact}

\begin{example} Define the firmly nonexpansive operator $T\colon X \to X$ by
\begin{equation*}
T := \tfrac{1}{2}\Id+\tfrac{1}{2}R, 
\end{equation*}
set $x_0 := e_1$, and 
$(x_n)_\nnn := (T^nx_0)_\nnn$ with $x_0 = e_1.$
Then the following hold:
\begin{enumerate}
\item\label{i:iterate} $(x_{n})_\nnn$ is Fej\'er monotone with respect to $\Fix T=\{0\}$, and
$x_n\to 0$. Moreover, $(\forall \nnn)$ $x_{n+1}\neq x_{n}\neq 0$. 
\item\label{i:perp:inf}
$(\forall \nnn)\ \scal{x_{n+1}}{x_{n}-x_{n+1}}=0.$
\item \label{i:weakcon} Both
$$\Big(\frac{x_n}{\|x_n\|}\Big)_{\nnn} \;\;\text{ and }\;\; \Big(\frac{x_n-x_{n+1}}{\|x_n-x_{n+1}\|}\Big)_\nnn$$
converge weakly --- but not strongly --- to $0\not\in \sss $.
\end{enumerate}
\end{example}
\begin{proof}
\cref{i:iterate}: It is well known that $(x_{n})_\nnn$ is Fej\'er monotone with respect to $\Fix T=\{0\}$, because $T$ is nonexpansive. Since $R^ke_{1}=e_{k+1}$, we have
\begin{equation}\label{e:xn}
x_n = T^n x_{0}=\frac{1}{2^n}(\Id+R)^{n}x_{0}=\frac{1}{2^n}\sum_{k=0}^{n}\binom{n}{k}R^k x_{0}=\frac{1}{2^n}
\sum_{k=0}^n \binom{n}{k} e_{k+1}.
\end{equation}
Hence, by \cref{l:vander},
\begin{equation*}
\|x_n\|^2 = \frac{1}{4^n}\sum_{k=0}^n \binom{n}{k}^2
= \frac{1}{4^n}\binom{2n}{n} =
\frac{1}{4^n}\frac{(2n)!}{(n!)^2};
\end{equation*}
in particular, $x_{n}\neq 0$. $x_{n+1}\neq x_{n}$ because $x_{n+1}$ contains a nonzero term of
$e_{n+2}$ and $e_{n+2}\perp e_{k+1}$ for $1\leq k\leq n$.
 
Now recall \emph{Stirling's formula} (see, e.g., \cite[Theorem~5.44]{Stromberg})
which states that
\begin{equation}\label{e:stirling}
	n! \approx \sqrt{2\pi n}\frac{n^n}{e^n}
\end{equation}
for large $n$,
and which implies
\begin{equation*}
\|x_n\|^2 = \frac{1}{4^n}\binom{2n}{n} = \frac{(2n)!}{4^n(n!)^2} \approx \frac{1}{4^n}
\frac{\sqrt{2\pi(2n)}(2n/e)^{2n}}{(\sqrt{2\pi n})^2(n/e)^{2n}}
= \frac{1}{\sqrt{\pi}}\frac{1}{\sqrt{n}}\to 0.
\end{equation*}
(The qualitative fact that $x_n\to 0$ also follows from \cite[Example~5.29]{BC2017} or
\cite[Corollary 1.2]{bruck77}.)

\cref{i:perp:inf}: Since $R$ is an isometry, this follows from 
\cref{l:skew}\cref{i:resol2}.

\cref{i:weakcon}. For \emph{fixed} $k$, we have from Stirling's formula \cref{e:stirling} that
\begin{equation}\label{e:nchoose}
\binom{n}{k} = \frac{n!}{k!(n-k)!}
\approx
\frac{1}{k!}\frac{\sqrt{2\pi n}(n/e)^{n}}{\sqrt{2\pi (n-k)}((n-k)/e)^{n-k}}
\approx \frac{n^k}{k!}
\end{equation}
for large $n$.
Hence
\begin{align*}
\frac{x_n}{\|x_n\|}
&\approx
\sqrt[4]{\pi}\sqrt[4]{n}
\frac{1}{2^n}\sum_{k=0}^n \binom{n}{k} e_{k+1}
\approx \sum_{k=0}^n
\frac{\sqrt[4]{\pi}\sqrt[4]{n}}{2^n} \frac{n^k}{k!} e_{k+1} \rightharpoonup 0
\end{align*}
because that $(e_{k})_{\kkk}$ is a total set in $\ell^2(\NN)$ and that
for each fixed $\kkk$ the coefficient of $e_{k+1}$ in $x_{n}/\|x_{n}\|$ clearly converges to $0$ as
$n\to\infty$; see, e.g., \cite[Example 4.8-6]{kreyszig}.

Next,
\begin{align}\label{e:diff}
x_{n}-x_{n+1}
&=
\frac{1}{2^n}\sum_{k=0}^n \binom{n}{k} e_{k+1}
- \frac{1}{2^{n+1}}\sum_{k=0}^{n+1} \binom{n+1}{k} e_{k+1}.
\end{align}
Since $e_{n+2}\perp e_{k+1}$ for $0\leq k\leq n$, by \cref{l:vander} and \cref{e:xn}
we have
\begin{subequations}
\label{e:simp}
\begin{align}
\scal{x_{n}}{x_{n+1}} &=\scal{\frac{1}{2^n}\sum_{k=0}^n \binom{n}{k} e_{k+1}}{\frac{1}{2^{n+1}}\sum_{k=0}^{n+1} \binom{n+1}{k} e_{k+1}}\\
&=\frac{1}{2}\frac{1}{4^n}\scal{\sum_{k=0}^n \binom{n}{k} e_{k+1}}{\sum_{k=0}^{n} \binom{n+1}{k} e_{k+1}}\\
&=\frac{1}{2}\frac{1}{4^n}\sum_{k=0}^n \binom{n}{k}\binom{n+1}{k}=\frac{1}{2}\frac{1}{4^n}\sum_{k=0}^n \binom{n}{n-k}\binom{n+1}{k}=\frac{1}{2}\frac{1}{4^n}\binom{2n+1}{n}.
\end{align}
\end{subequations}

It follows from \cref{e:xn} and \cref{e:simp} that
\begin{subequations}
\label{e:diff:norm}
\begin{align}
\|x_n-x_{n+1}\|^2 &=\|x_{n}\|^2+\|x_{n+1}\|^2-2\scal{x_{n}}{x_{n+1}}\\
& =\frac{1}{4^{n}}\binom{2n}{n}
+\frac{1}{4^{n+1}}\binom{2(n+1)}{n+1}-2\frac{1}{2}\frac{1}{4^n}\binom{2n+1}{n}
\\
&=\frac{1}{4^n}\bigg[\binom{2n}{n}+\frac{1}{4}\binom{2n+2}{n+1}-\binom{2n+1}{n}\bigg]\\
& =\frac{1}{2(n+1)4^n}\binom{2n}{n}
= \frac{1}{2(n+1)}\|x_n\|^2\\
&\approx
\frac{1}{2n}\frac{1}{\sqrt{\pi n}}
= \frac{1}{2\sqrt{\pi}}\frac{1}{n^{3/2}}
\end{align}
\end{subequations}
for large $n$.
Combining \cref{e:diff}, \cref{e:diff:norm}, and \cref{e:nchoose}, we obtain
\begin{subequations}
\begin{align}
\frac{x_n-x_{n+1}}{\|x_n-x_{n+1}\|}
&\approx
\frac{\sqrt{2}\sqrt[4]{\pi}n^{3/4}}{2^n}\sum_{k=0}^n \binom{n}{k} e_{k+1}
- \frac{\sqrt{2}\sqrt[4]{\pi}n^{3/4}}{2^{n+1}}\sum_{k=0}^{n+1} \binom{n+1}{k} e_{k+1}\\
&\approx
\sum_{k=0}^n \frac{\sqrt{2}\sqrt[4]{\pi}n^{3/4}}{2^n}\binom{n}{k} e_{k+1}
- \sum_{k=0}^{n+1}\frac{\sqrt{2}\sqrt[4]{\pi}n^{3/4}}{2^{n+1}} \binom{n+1}{k} e_{k+1}\\
&\approx
\sum_{k=0}^n \frac{\sqrt{2}\sqrt[4]{\pi}n^{3/4}}{2^n}\frac{n^k}{k!} e_{k+1}
- \sum_{k=0}^{n+1}\frac{\sqrt{2}\sqrt[4]{\pi}n^{3/4}}{2^{n+1}}\frac{(n+1)^k}{k!} e_{k+1}
\\
&\rightharpoonup 0,
\end{align}
\end{subequations}
because for every fixed $\kkk$ the coefficients of $e_{k+1}$ in
$(x_{n}-x_{n+1})/\|x_{n}-x_{n+1}\|$ converge to $0$
as $n\to\infty$.

In summary, both quotient limits converge weakly but not strongly to $0$.
\end{proof}

\section*{Acknowledgements}
We thank Terry Rockafellar for his inspirational talk at the virtual
West Coast Optimization Meeting in May 2021 which stimulated this research. 
We also thank Walaa Moursi for suggesting to investigate the operator $T$ in \cref{sec:shift}.
HHB and XW were partially supported by NSERC Discovery Grants.		
MKL was partially supported by HHB and XW's NSERC Discovery Grants.


\addcontentsline{toc}{section}{References}

\bibliographystyle{abbrv}

\end{document}